\documentclass[11pt,reqno]{amsart}
\usepackage[utf8]{inputenc}
\usepackage{amsmath, amsfonts, amsthm, amssymb, verbatim, enumitem, mathtools,tikz,tikz-cd} 
\usepackage{bbold}

\usepackage{tikz}
\usepackage{tikz-cd}
\usepackage{amssymb}
\usepackage{amsgen}
\usepackage{amsmath}
\usepackage{amsthm}
\usepackage{mathrsfs}
\usepackage{cite}
\usepackage{amsfonts}

\newcommand{\pdim}{\mathop{\mathrm{pd}}\nolimits}
\newcommand{\gldim}{\mathop{\mathrm{gl.dim}}\nolimits}

\newcommand{\Inf}{\mathop{\mathrm{Inf}}\nolimits}

\newcommand{\inv}{^{-1}}
\newcommand{\p}{\varphi}

\newcommand{\ov}[1]{\ensuremath{\overline {#1}}}
\newcommand{\til}[1]{\ensuremath{\widetilde {#1}}}

\newcommand{\Hom}{\mathop{\mathrm{Hom}}\nolimits}
\newcommand{\Ext}{\mathop{\mathrm{Ext}}\nolimits}
\newcommand{\Tor}{\mathop{\mathrm{Tor}}\nolimits}

\usepackage{xcolor}

\newtheorem{Thm}{Theorem}[section]
\newtheorem{Prop}[Thm]{Proposition}

\newtheorem{Lemma}[Thm]{Lemma}
{\theoremstyle{definition}
}
{\theoremstyle{remark}
\newtheorem{Rmk}[Thm]{Remark}}
\newtheorem{Cor}[Thm]{Corollary}
{\theoremstyle{remark}
}
{\theoremstyle{remark}
}

{\theoremstyle{remark}
}
{\theoremstyle{remark}
}
{\theoremstyle{remark}
}
{\theoremstyle{remark}
\newtheorem*{Claim*}{Claim}}
\newtheorem*{thmA}{Theorem~A}
\newtheorem*{thmB}{Theorem~B}
\newtheorem*{thmC}{Theorem~C}

\numberwithin{equation}{section}

\usepackage{multirow, multicol, graphicx, color, setspace} 
\usepackage[margin=1in]{geometry}

\title[Moore spaces as classifying spaces of finite semigroups]{Realizing wedges of Moore spaces as classifying spaces of finite semigroups}

\author{Aris Martinian}
\address[A.~Martinian]{%
    Department of Mathematics\\
    City College of New York\\
    Convent Avenue at 138th Street\\
    New York, New York 10031\\
    USA}
\email{arismartinian@gmail.com}

\author{Benjamin Steinberg}
\address[B.~Steinberg]{%
    Department of Mathematics\\
    City College of New York\\
    Convent Avenue at 138th Street\\
    New York, New York 10031\\
    USA}
\email{bsteinberg@ccny.cuny.edu}

\thanks{The first author was supported by NSF DMS-2452324,  Simons Foundation Collaboration Grant 849561, Australian Research Council Grant DP230103184, and Marsden Fund Grant MFP-VUW2411.}
\date{November 16, 2025}

\dedicatory{To Mikhail Volkov on his seventieth birthday}

\begin{document}

\begin{abstract}
Fiedorowicz suggested that it was likely that every finite simply connected CW complex is homotopy equivalent to the classifying space of a finite semigroup.  We prove that every finite wedge of simply connected Moore spaces of finitely generated abelian groups is homotopy equivalent to the classifying space of a finite semigroup.  Consequently, homology groups alone cannot preclude a finite simply connected CW complex from being homotopy equivalent to the classifying space of a finite semigroup.
\end{abstract}

\maketitle
\section{Introduction}
A CW complex is homotopy equivalent to the classfying space of a group if and only if it is aspherical.  It was shown by Dusa McDuff~\cite{McDuff} that every connected CW complex is homotopy equivalent to the classifying space of a monoid (see~\cite{Fiedor} for an alternate proof).  McDuff's construction invariably yields an infinite monoid, and in fact the fundamental group of the classifying space of a finite monoid must be finite.

In his paper~\cite{Fiedor}, Z. Fiedorowicz observed that the $2$-sphere is homotopy equivalent to the classifying space of a $5$-element monoid, and wrote:
\begin{quote}
It seems likely that any finite simply connected complex
should be equivalent to the classifying space of a finite monoid.
\end{quote}

In this paper we build upon some results by Sweeney~\cite{homotopysmall} to make some significant progress on this conjecture.  Note that if $X$ is a finite simply connected CW complex, then $H_0(X)\cong \mathbb Z$, $H_1(X)=0$, $H_n(X)$ is finitely generated for $n\geq 2$, and moreover $H_n(X)=0$ if $n> \dim X$.  Our main theorem shows that we can realize all possible homology groups of simply connected CW complexes from a finite monoid.  More precisely, we prove:

\begin{thmA}
Let $A_2,A_3,\ldots, A_m$ be a finite sequence of finitely generated abelian groups.  Then there is a finite monoid $M$ with simply connected classifying space $BM$ such that $H_n(BM)\cong A_n$ for $2\leq n\leq m$, and $H_n(BM)=0$ for $n>m$. 
\end{thmA}

Theorem~A is obtained as a corollary of a more precise result.  Recall that if $A$ is an abelian group and $n\geq 2$, then a CW complex is a Moore space of type $M(A,n)$ if it is simply connected and its reduced homology is $A$ in degree $n$ and $0$ in all other degrees.  An $M(A,n)$ is unique up to homotopy equivalence.  A wedge sum of Moore spaces $M(A_n,n)$ with $2\leq n\leq m$ is a finite simply connected CW complex that has the homology indicated in Theorem~A.  A simply connected CW complex is homotopy equivalent to a wedge of Moore spaces if and only if the Hurewicz homomorphisms $h_n\colon \pi_n(X,x_0)\to H_n(X)$ are split epimorphisms for all $n\geq 2$~\cite[Proposition~2.6.15]{baueshomotopy}.  Thus Theorem~A is a corollary of the following Theorem~B.

\begin{thmB}
If $X$ is a finite simply connected CW complex such that the Hurewicz homomorphisms $h_n\colon \pi_n(X,x_0)\to H_n(X)$ are split epimorphisms for all $n\geq 2$, that is, $X$ is homotopy equivalent to a finite wedge of Moore spaces of finitely generated abelian groups, then $X$ is homotopy equivalent to the classifying space of a finite monoid.
\end{thmB}

%constructing, for every $n\geq 1$, a finite semigroup $S_n$ whose classifying space $BS_n$ is a Moore space $M(C_n,2)$. In Section 1, we detail the construction and compute homology. In Section 2, we obtain $M(G,2)$ for an arbitrary finitely generated abelian group $G$ via the wedge sum. Iterated suspension (Section 3) then yields $M(G,k)$ for $k\geq 2$. Finally, wedge sums of these Moore spaces realize every finitely  generated graded abelian group, vanishing in degrees zero and one, as the reduced homology of a simply connected CW-complex.

A natural strategy for proving Theorem~B is to first show the closure (up to homotopy equivalence) of the classifying spaces of finite monoids under suspension and wedge sum, and then to realize an $M(C,2)$ as the classifying space of a finite monoid for any cyclic group $C$.  Realizing suspensions has already been accomplished in~\cite{homotopysmall}, where an $M(C_2,2)$ classifying space of a finite monoid was also constructed.  Note that $S^2$ is homotopy equivalent to the classifying space of a finite monoid by~\cite{Fiedor}, and so we also have $M(\mathbb Z,2)$.  The key problem with this approach is that the natural way to realize a wedge sum of classifying spaces of monoids is via the free product construction, which immediately takes you out of the class of finite monoids.  

The key new contributions in this paper are the following.  We show that if $M$ and $N$ are finite monoids whose minimal ideals are rectangular bands, then the wedge sum $BM\vee BN$ is homotopy equivalent to the classifying space of a finite monoid whose minimal ideal is a rectangular band.  Sweney's construction of the suspension~\cite{homologicalepi} always produces a monoid whose minimal ideal is a rectangular band.  It then remains to construct a finite monoid whose minimal ideal is a rectangular band, and whose classifying space is an $M(C,2)$ for each cyclic group $C$, which we do.

We end the paper by answering some questions of Sweeney~\cite{homotopysmall} on the homology of bands and finite von Neumann regular monoids.   In particular, we completely answer Question~8.6 and answer most of Question~8.4 of~\cite{homotopysmall}.  We also determine when the algebra of a finite regular monoid over an arbitrary commutative ring has finite global dimension, generalizing the results of Nico~\cite{Nico1} over a field:

\begin{thmC}
Let $M$ be a finite von Neumann regular monoid and $K$ a commutative ring.  Then $KM$ has finite global dimension if and only if $K$ has finite global dimension and the order of every maximal subgroup of $M$ is a unit in $K$.
\end{thmC}

Like Sweeney, we work with semigroups throughout, although there is no real difference between working with semigroups or monoids.

\section{Preliminaries}
This section considers the preliminaries we need from semigroup theory, topology and homological algebra.

\subsection{Semigroup theoretic preliminaries}
A \emph{semigroup} is a set $S$ with an associative binary operation.  We shall tacitly assume that $S$ is nonempty.  A \emph{monoid} is a semigroup with identity, and any semigroup $S$ can be turned into a monoid $S^{\mathbb 1}$ by adjoining an identity.  The reader is referred to~\cite{CP} for basic semigroup theory and the appendix of~\cite{qtheor} for specifics on finite semigroup theory.

A semigroup is a \emph{band} if each of its elements is idempotent.  A rectangular band is a band satisfying the identity $xyx=x$.  Up to isomorphism, every rectangular band is obtained by taking nonempty sets $A,B$ and endowing $A\times B$ with the product \[(a,b)(a',b') = (a,b')\] whence the name ``rectangular.''

An \emph{ideal} of a semigroup $S$ is a nonempty subset $I$ such that $SI\cup IS\subseteq I$.  The intersection of two ideals is an in ideal, and hence every finite semigroup contains a unique minimal ideal.  A semigroup is simple if it has no proper ideals.  A rectangular band is simple and the minimal ideal of a finite semigroup is simple.  The  Rees-Suschkewitsch describes all finite simple semigroups up to isomorphism.  

A semigroup $S$ may have a zero element $z$, that is $zs=z=sz$ for all elements $s\in S$.  In this case $\{z\}$ is the minimal ideal of $S$, and an ideal $I$ of $S$ is called \emph{$0$-minimal} if it is a minimal nonzero ideal.  Note that either $I$ is null, or $I$ is $0$-simple.   A semigroup $S$ with zero is \emph{null} if $S^2=0$ and is \emph{$0$-simple}  if it is not null and contains no proper nonzero ideal.  Rees's theorem describes all finite $0$-simple semigroups up to isomorphism. 
If $I$ is a $0$-minimal $0$-simple ideal of a finite semigroup $S$, then the nonzero principal left ideals of $I$ are minimal nonzero left ideals of $S$, and any two of them intersect in $0$.  Moreover, every principal left ideal of $I$ can be generated by an idempotent.  If $S$ is any semigroup, we can always adjoin a zero element to obtain a semigroup $S^{\mathbb 0}$ with zero.  
If $I$ is an ideal of a semigroup $S$, the Rees quotient $S/I$ is the quotient of $S$ by the congruence identifying all of the elements of $I$ (and nothing else).  Note that $S/I$ is a semigroup with zero element the class of $I$.

A \emph{principal series} for a semigroup $S$ is an unrefinable chain of two-sided ideals
\begin{equation}\label{eq:principal.series}    
I_n\subsetneq \cdots\subsetneq I_1\subsetneq I_0=S.
\end{equation}
In particular, $I_n$ must be the minimal ideal of $S$.  Note that every finite semigroup admits a principal series.  If $s\in I_j\setminus I_{j+1}$, then $I_j = S^{\mathbb 1}sS^{\mathbb 1}\cup I_{j+1}$, where we set $I_{n+1}=\emptyset$ for convenience. Moreover, $I_j/I_{j+1}$ is either $0$-simple or null.  

A semigroup $S$ is called (von Neumann) \emph{regular} if, for all $s\in S$, there exists $t\in S$ with $sts=s$.  A semigroup is regular if and only if each principal left (respectively, right) ideal can be generated by an idempotent.  A finite semigroup $S$ is regular if and only if each principal two-sided ideal can be generated by an idempotent.  Every band is a regular semigroup, as is every group.  Any $0$-simple or simple semigroup is regular.
Given a principal series \eqref{eq:principal.series} in a finite regular semigroup $S$, we can always find an idempotent $e_j\in S$ with $I_j = Se_jS\cup I_{j+1}$ and, in particular, $I_j/I_{j+1}$ is $0$-simple.

Every finite semigroup contains an idempotent.
If $e$ is an idempotent of $S$, then $eSe$ is a monoid with identity $e$.  The group of units of $eSe$ is denoted $G_e$ and is called the \emph{maximal subgroup} of $G$ at $e$.  If each maximal subgroup of $S$ is trivial, then $S$ is said to be \emph{aperiodic}. For example, every band is aperiodic.  In a finite semigroup, idempotents that generate the same principal ideal have isomorphic maximal subgroups.   If $I$ is a finite $0$-simple semigroup and $e$ is a nonzero idempotent, then $eIe=G_e\cup \{0\}$.

If $S$ is a semigroup, then the \emph{group completion} $G(S)$ of $S$ is the universal group receiving a homomorphism from $S$, i.e., the left adjoint of the forgetful functor from groups to semigroups.  One way to construct $G(S)$ is to take a presentation for $S$ as a semigroup and view it as a group presentation.  If $S$ is a finite semigroup, then $G(S)$ is a homomorphic image of $S$, and hence is finite.  In fact, $G(S)$ is a certain homomorphic image of the maximal subgroup of the minimal ideal of $S$; see~\cite{Arbib} or~\cite[Theorem~6.3]{homotopysmall} for details.

If $S$ is a semigroup and $K$ is a commutative ring (always assume to have a unit), then the \emph{semigroup algebra} $KS$ is the (not necessarily unital) $K$-algebra with basis $S$ and the unique bilinear multiplication extending the multiplication of $S$. If $S$ is a monoid, then $KS$ is unital, of course.

If $S$ is a semigroup with zero element $z$ and $K$ is a commutative ring, then the \emph{contracted semigroup algebra} of $S$ is $K_{\mathbb 0}S=KS/Kz$.   So $K_{\mathbb 0}S$ has basis $S\setminus \{z\}$ with product extending that of $S$ where we identify the zeroes of $S$ and its contracted algebra.  Notice that $KS\cong K_{\mathbb 0}[S^{\mathbb 0}]$.  One observes directly that if $I$ is an ideal of $S$, then $KS/KI\cong K_{\mathbb 0}[S/I]$.  Also, if $S$ has a zero, then $K_{\mathbb 0}S/K_{\mathbb 0}I\cong  K_{\mathbb 0}[S/I]$.

If $M$ is a monoid, then a left $M$-set is a set $X$ equipped with an action $M\times X\to X$, written $(m,x)\mapsto mx$ satisfying $1x=x$ and $m(m'x)=(mm')x$ for all $m,m'\in M$ and $x\in X$. Right $M$-sets are defined dually.  A map $f\colon X\to Y$ of $M$-sets is \emph{equivariant} if $f(mx)=mf(x)$ for all $m\in M$ and $x\in X$.  An $M$-set $X$ is \emph{free} if it has a basis, that is, there is a subset $B\subseteq X$ such that each $x\in X$ is uniquely of the form $x=mb$ with $m\in M$ and $b\in B$.  Free $M$-sets have the usual universal property.  An $M$-set $P$ is \emph{projective} if every equivariant surjective map $f\colon X\to P$ has an equivariant section.  It is well known~\cite{actsbook} that every projective $M$-set is isomorphic to one  of the form $\bigsqcup_{a\in A} Me_a$ where the $e_a$ are not necessarily distinct idempotents.

If $X$ is a right $M$-set and $Y$ is a left $M$-set, then $X\otimes_M Y = (X\times Y)/{\sim}$ where $\sim$ is the least equivalence relation so that $(xm,y)\sim (x,my)$ for all $x\in X$, $y\in Y$ and $m\in M$.  The equivalence class of $(x,y)$ is denoted $x\otimes y$.  If $X$ is an $N$-$M$-biset, i.e., has a left action of $N$ that commutes with the right action of $M$, then $X\otimes_M Y$ is a left $N$-set.  If $X$ is a free right $M$-set with basis $B$, then $X\otimes_M Y= \bigsqcup_{b\in B}b\otimes Y\cong \bigsqcup_{b\in B} Y$.  If $X=eM$ with $e$ an idempotent, then $eM\otimes_M Y\cong eY$ as a left $eMe$-set via the multiplication map.

The trivial left (respectively, right) $KS$-module has underlying $K$-module $K$, where each element of $S$ acts on $K$ as the identity.  Often, the trivial $\mathbb ZS$-module will be called the trivial module for $S$.

\subsection{Topological preliminaries}
The reader  is referred to~\cite{may} for background on simplicial sets and to~\cite{Hatcher} for $\Delta$-sets.  Associated to every semigroup $S$ is a $\Delta$-set whose set of $q$-simplices is $S^q$.  The face maps are given by 
\begin{align*}
d_0(s_1,\ldots,s_q) & = (s_2,\ldots,s_q),\\
d_i(s_1,\ldots, s_q) & =(s_1,s_2,\ldots,s_{i-1},s_is_{i+1}, s_{i+2},\ldots,s_q),\quad \text{for}\ 1\leq i\leq q-1,\\
d_q(s_1,\ldots, s_q) & = (s_1,\ldots, s_{q-1}).
\end{align*}
The geometric realization of this $\Delta$-set is denoted by $BS$ and is called the \emph{classifying space of $S$}.  We will often not distinguish between a $\Delta$-set and its geometric realization.   In the case that $S$ is a monoid, we can give this $\Delta$-set the structure of a simplicial set by having the $i^{th}$-degeneracy insert a $1$ at position $i$.  So the degenerate simplices are those simplices containing a $1$. Since the fat geometric realization of a simplicial set, i.e., the realization of its underlying $\Delta$-set, is homotopy equivalent to its geometric realization as a simplicial set, for our purposes it doesn't matter whether we view the classifying space of a monoid as coming from its $\Delta$-set or its simplicial set.  Note that if $S$ is a semigroup, then the geometric realization of the simplicial set associated to $S^{\mathbb 1}$ is precisely $BS$ and hence $BS\simeq BS^{\mathbb 1}$, as was observed in~\cite[Proposition~4.1]{homotopysmall}.  Note that if $M,N$ are monoids, at least one of which is countable, then $B(M\times N)\cong BM\times BN$ (otherwise one must take the product in the category of compactly generated spaces)~\cite[Proposition~4.3]{homotopysmall}.  An important fact for us is that $\pi_1(BS)\cong G(S)$~\cite[Lemma~1]{McDuff}.  

If $X,Y$ are pointed spaces, then the \emph{wedge sum} $X\vee Y$ is obtained from the disjoint union by identifying the base points.  For CW complexes, we assume that the basepoints are vertices so that we obtain a CW complex.  It is well known (cf.~\cite{Fiedor}) that if $M,N$ are monoids, then $B(M\ast N)\simeq BM\vee BN$, where $M\ast N$ is the free product, but the free product of nontrivial monoids is never finite.  

If $X$ is a topological space, then the \emph{suspension} of $X$ is the quotient of $X\times [0,1]$ obtained by identifying $X\times \{0\}$ to a point and $X\times \{1\}$ to a different point.  For example, $\Sigma S^n\cong S^{n+1}$, where $S^k$ is the $k$-sphere.  If $X$ is path connected, then $\Sigma X$ is simply connected and $H_{n+1}(\Sigma X)\cong H_n(X)$ for all $n\geq 1$. If $X$ is a CW complex, then $\Sigma X$ is naturally a CW complex. 

If $n\geq 2$ and $A$ is an abelian group, a Moore space of type $M(A,n)$ is a simply connected CW complex $X$ with $H_n(X)\cong A$ and $\til H_q(X)=0$ for $q\neq n$.  For example, the $n$-sphere $S^n$ is a Moore space of type $M(\mathbb Z,n)$.  It is well known that any two Moore spaces of type $M(A,n)$ are homotopy equivalent~\cite{Hatcher}.  Note that if $X$ is an $M(A,n)$, then $\Sigma X$ is an $M(A,n+1)$.  

We shall also need the machinery of equivariant classifying spaces from~\cite{TopFinite1}.  If $M$ is a monoid, a (left) $M$-space is a topological spaces $X$ with a continuous $M$-action $M\times X\to X$, where $M$ has the discrete topology.  Right $M$-spaces are defined dually.  Tensor products of $M$-spaces are given the quotient topology.

Let $B^n$ be the $n$-dimensional closed Euclidean ball.  A \emph{projective} $n$-cell is a space of the form $Me\times B^n$ where $e\in M$ is an idempotent, and $Me$ is given the discrete topology. This is an $M$-space with action $m(x,t)=(mx,t)$ for $x\in Me$ and $t\in T$.  An $M$-CW complex $X$ is an inductive limit of $M$-spaces $X^k$ where $X^{k+1}$ is obtained form $X^k$ by attaching a collection $\bigsqcup_{a\in A} Me_a\times B^{k+1}$ of projective $M$-cells via an $M$-equivariant continuous mapping $\bigsqcup_{a\in A}Me_a\times S^k\to X^k$.  A contractible projective $M$-CW complex is called an \emph{equivariant classifying space} for $M$, and they are unique up to $M$-homotopy equivalence.  If $X$ is an equivariant classifying space for $M$, then the augmented cellular chain complex of $X$ provides a projective resolution of the trivial $\mathbb ZM$-module.  

If $X$ is a projective left $M$-CW complex and $A$ is an $N$-$M$-biset that is projective as a left $N$-set, then $A\otimes_M X$ is a projective $N$-CW complex~\cite[Corollary~3.2]{TopFinite1}.  In particular, taking $N$ to be the trivial monoid, if $X$ is a projective $M$-CW complex, then $M\backslash X:=\{1\}\otimes_M M$ is a CW complex.  It is shown in~\cite[Corollary~6.7]{TopFinite1} that if $X$ is any equivariant classifying space for $M$, then $M\backslash X\simeq BM$.  Note that there is an equivariant classifying space $EM$ such that $M\backslash EM=BM$.

\subsection{Homological preliminaries}
Let $R$ be a ring with unit.  We shall work with left modules as the default, but, of course, there are dual notions for right modules.  The \emph{projective dimension} $\pdim M$ of an $R$-module $M$ is the length of the shortest projective resolution of $M$, which might be infinite.  The (left) \emph{global dimension} of $R$ is the supremum of the projective dimensions of all (left) $R$-modules.  

We shall use throughout the following well-known fact, cf.~\cite[Lemma VIII.2.1]{Browncohomology}.

\begin{Lemma}\label{l:brown}
Let $M$ be an $R$-module.  Then the following are equivalent:
\begin{enumerate}
    \item $\pdim M\leq n$;
    \item $\Ext_R^i(M,-)=0$ for all $i>n$;
    \item $\Ext^{n+1}(M,-)=0$;
    \item If $0\to K\to P_{n-1}\to \cdots P_0\to M\to 0$ is an exact sequence of $R$-modules with each $P_i$ projective, then $K$ is projective.
\end{enumerate}
\end{Lemma}

If $\p\colon R\to S$ is a homomorphism of unital rings, then any $S$-module $A$ can be viewed as an $R$-module, which we call the \emph{inflation} of $A$ along $\p$.

If $M$ is a monoid, then the projective dimension of the (left/right) trivial module $\mathbb Z$ is often called the (left/right) \emph{cohomological dimension} of $M$.

 If $A$ is a left $\mathbb ZM$-module, then the homology of $M$ with coefficients in $A$ is $H_n(M,A) = \Tor_n^{\mathbb ZM}(\mathbb Z,A)$ and the cohomology of $M$ with coefficients in $A$ is $H^n(M,A) = \Ext^n_{\mathbb ZM}(\mathbb Z,A)$, where $\mathbb Z$ has the trivial module structure.  If $B$ is a right $\mathbb ZM$-module, we may also write $H_n(M,B)$ for $\Tor_n^{\mathbb ZM}(B,\mathbb Z)$. Hopefully, no confusion will arise.   Note that any abelian group $A$ can be viewed as a left/right trivial $\mathbb ZM$-module, i.e., one where elements of $M$ act as the identity on the appropriate side.    We put $H_n(M)=H_n(M,\mathbb Z)$ and $H^n(M)=H^n(M,\mathbb Z)$ where $\mathbb Z$ has the trivial module structure.

Of primary importance to us is the following classical isomorphism, cf.~\cite{GabrielZisman}.  See also~\cite[Theorem~A.1]{homotopysmall}.

\begin{Thm}\label{t:top.interp.homology}
Let $M$ be a monoid and $A$ any $\mathbb Z\pi_1(BM)$-module, inflated to a $\mathbb ZM$-module.  Then $H_n(BM,A)\cong H_n(M,A)$.  This applies, in particular, to any abelian group $A$ with the trivial $\mathbb ZM$-module structure.
\end{Thm}

\section{Suspensions and wedge sums}
It was shown in~\cite[Corollary~9.4]{homotopysmall} that if $S$ is a monoid, then there is a monoid $J(S)$ with $BJ(S)\simeq \Sigma (BS)$ using Brown-Forman Discrete Morse theory.  The monoid $J(S)$ is $S\cup K$ where $K$ is the rectangular band $\{1,2\}\times S$, which is the minimal ideal of $J(S)$.  For $s\in S$, $s'\in S$ one has $s(i,s') = (i,s')$ and $(i,s')s=(i,s's)$.  In particular, if the homotopy type of a CW complex $X$ can be realized by a finite semigroup, then the homotopy type of $\Sigma X$ can be realized by a finite monoid whose minimal ideal is a rectangular band.  We will only need to apply this construction to the case where $BS$ is a Moore space of type $M(A,n)$ with $n\geq 2$ and $A$ is a finitely generated abelian group.  In this case $\Sigma(BS)$ will be a Moore space of type $M(A,n+1)$, and to prove that $BJ(S)$ is homotopy equivalent to $\Sigma(BS)$, we just need that it has the same homology as $\Sigma(BS)$.  We shall give an algebraic proof of this fact in Section~\ref{s:globdim}, below, in order to be self-contained.

It is well known, cf.~\cite{Fiedor}, that if $M$ and $N$ are monoids, then $B(M\ast N)\simeq BM\vee BN$.  But, unfortunately, the free product of nontrivial monoids is never finite.  We now show that if $X$ and $Y$ are simply connected CW complexes whose homotopy types can be realized by finite monoids with rectangular bands as minimal ideals, then $X\vee Y\simeq BS$ where $S$ is a finite monoid whose minimal ideal is a rectangular band.  We prove the following more precise result.

\begin{Thm}\label{t:wedge}
Let $M$ be a monoid whose minimal ideal $K$ is a rectangular band and let $N$ be any monoid.  Let  $S$ be the submonoid $(M\times \{1\})\cup (K\times N)$ of $M\times N$.  Then $BS\simeq BM\vee BN$.   Moreover, if $N$ has minimal ideal $J$, then $S$ has minimal ideal $K\times J$, and $S$ is finite if $M$ and $N$ are finite.
\end{Thm}
\begin{proof}
Let us assume that $K=A\times B$.  Fix $b_0\in B$ and put $e_a=(a,b_0)$ for $a\in A$.  Fix also $a_0\in A$.  We identify $M$ with the submonoid $M\times \{1\}$ of $S$.   By~\cite[Lemma~6.4]{TopFinite1} we can choose an equivariant classifying space $X$ for $M$ such that $X^0\cong Me_{a_0}=A\times \{b_0\}$ and an equivariant classifying space $Y$ for $N$ such that $Y^0\cong N$.   Note that $S(e_{a_0},1) = A\times \{b_0\}\times N$ is a free right $N$-set with basis $\{(a,b_0,1): a\in A\}$.  Thus $S(e_{a_0},1)\otimes_N Y$ is a projective $S$-CW complex~\cite{TopFinite1}, and $S(e_{a_0},1)\otimes_N Y =\bigsqcup_{a\in A}(e_a,1)\otimes Y\cong \bigsqcup_{a\in A} Y$.  Note that the vertex set is $\{(e_a,1)\otimes n: a\in A,n\in N\}$, which is in bijection with $A\times N$, and is isomorphic $S(e_{a_0},1)$ as a left $S$-set.

Next we observe that $S$ is a projective right $M$-set.  Indeed, we have $S=M\sqcup (K\times N\setminus \{1\})$ where $M$ acts trivially $N\setminus\{1\}$.  But $K$ is a projective right $M$-set because $K=\bigsqcup_{a\in A}e_aM$.  Thus $S\otimes_M X$ is a projective $S$-CW complex.  Moreover, \[S\otimes_M X = (1\otimes X)\sqcup \bigsqcup_{(a,n)\in A\times N\setminus\{1\}} (e_a,n)\otimes e_aX\cong X\sqcup ((N\setminus \{1\})\times \bigsqcup_{a\in A} e_aX)\] where $N\setminus \{1\}$ has the discrete topology.   The vertex set consists of $1\otimes Me_{a_0}$ and the elements $(e_a,n)\otimes e_a$ with $a\in A$ and $n\in N\setminus \{1\}$.  This is in bijection with $A\times N$ and is isomorphic to $S(e_{a_0},1)$ as a left $S$-set by sending $(a,b_0,n)$ to $1\otimes e_a$ if $n=1$, and to $(e_a,n)\otimes e_a$ when $n\neq 1$.

It follows that $Z'=(S\otimes_M X)\sqcup (S(e_{a_0},1)\otimes_N Y)$ is a projective $S$-CW complex.   Also, $Z= S(e_{a_0},1)\times I$ is a projective $S$-CW complex of dimension $1$, where $I$ is the unit interval, with vertex set $(S(e_{a_0},1)\times \{0\})\sqcup (S(e_{a_0},1)\times \{1\})$.  By~\cite{TopFinite1}, $W=Z\bigsqcup_{Z^0} Z'$ is a projective $S$-CW complex where we equivariantly glue in $Z^0$ as follows.  Our mapping will take $S(e_{a_0},1)\times \{0\}$ to $S\otimes_M X$ and  $S(e_{a_0},1)\times \{1\}$ to  $S(e_{a_0},1)\otimes_N Y$. To do this, we send $((e_a,n),0)$ to  $1\otimes e_a$ if $n=1$, and to $(e_a,n)\otimes e_a$ if $n\neq 1$. Also, we send $((e_a,n),1)$ to $(e_a,1)\otimes n$.  We claim that $W$ is contractible.  Notice that $e_aX$ is a retract of the contractible space $X$, and hence is contractible.  We contract $1\otimes X$ to a single vertex that we call $\ast$, we contract $(e_a,n)\otimes e_aX$, where $n\neq 1$, to a vertex we denote $v_{a,n}$ and we contract  $(e_a,1)\otimes Y$ to a vertex we denote $w_a$ without changing the homotopy type of $W$.  What remains is a graph where there is an edge from $\ast$ to $w_a$ for all $a\in A$ and an edge from $v_{a,n}$ to $w_a$ for each $a\in A$ and $n\neq 1$.  This graph is evidently a rooted tree with root $\ast$, where $\ast$ has $|A|$ children and each child has $|N\setminus \{1\}|$ children.   We conclude that $W$ is contractible.

%We can now glue these two complexes together to obtain a projective $S$-CW complex $Z$ by identifying the vertex $(e_a,1)\otimes n$ of $S\otimes_N Y$ with $1\otimes e_a$ if $n=1$ and with $(e_a,n)\otimes e_a$ when $n\neq 1$ by~\cite{TopFinite1}.   Notice that $e_aX$ is a retract of the contractible space $X$, and hence is contractible.  We can contract each of the subcomplexes $1\otimes X$, $(e_a,n)\otimes e_aX$ and $(e_a,1)\otimes Y$ of $Z$ without changing the homotopy type.   Contracting $(e_a,1)\otimes Y$ identifies all the vertices $(e_a,1)\otimes n$ with $(e_a,1)\otimes 1$.  Now if $n\geq 1$, then $(e_a,1)\otimes n$ is identified with $(e_a,n)\otimes e_aX$.  On the other hand, contracting $1\otimes Y$ identifies all the vertices $1\otimes e_a$ to $1\otimes e_{a_0}$.  But $1\otimes e_a$ is identified with $(e_a,1)\otimes 1$.  It now follows that after contracting all these subcomplexes we are just left with the vertex $1\otimes e_{a_0}$, and so $Z$ is contractible.

Now $S\backslash Z\cong \{1\}\otimes_S Z$, where $\{1\}$ will always denote a singleton set with trivial action.  Since tensor product commutes with colimits, $S\backslash W$ is the result of gluing $\{1\}\otimes_S (S(e_{a_0},1)\times I)\cong I$ into
\begin{align*}
(\{1\}\otimes_S S\otimes _M X)\sqcup  (\{1\}\otimes S(e_{a_0},1)\otimes_N Y) & \cong (\{1\}\otimes_M X)\sqcup (\{1\}\otimes_N Y)\\ &\cong M\backslash X\sqcup N\backslash Y
\end{align*} with the edge connecting the unique vertices of $M\backslash X$ and $N\backslash Y$.  Thus $S\backslash W$ is the homotopy pushout of $M\backslash X$ and $N\backslash Y$ along their base points.  Now $M\backslash X$ is homotopy equivalent to $BM$ and $N\backslash Y$ to $BN$ by~\cite[Corollary~6.7]{TopFinite1}.  Moreover, these homotopy equivalences may be assumed to be cellular, and hence preserve base points as each of these complexes has a single vertex.  Therefore $S\backslash W$ is homotopy equivalent to the homotopy pushout of $BM$ and $BN$ along their basepoints.  But this is homotopy equivalent to $BM\vee BN$ by contracting the edge between the base points.

Trivially, $S$ is finite if $M$ and $N$ are finite.  If $N$ has minimal ideal $J$, then $K\times J$ is the minimal ideal of $S$ by construction.  This completes the proof.
\end{proof}

It follows using~\cite[Corollary~9.4]{homotopysmall} (or Proposition~\ref{p:susp.homology} below) and Theorem~\ref{t:wedge}, that if we can construct for each cyclic group $C$ a semigroup $S_C$ whose minimal ideal is a rectangular band and with $BS_C$ a Moore space of type $M(C,2)$, then every finite wedge of Moore spaces of finitely generated abelian groups is homotopy equivalent to the classifying space of a finite semigroup. % Note that a finite simply connected CW complex $X$ is homotopy equivalent to such a wedge of Moore spaces if and only if the Hurewicz homomorphisms $h_n\colon \pi_n(X,x_0)\to H_n(X)$ are split epimorphisms for all $n\geq 2$~\cite[Proposition~2.6.15]{baueshomotopy}.

The results of~\cite{Fiedor} show that the $2\times 2$ rectangular band has classifying space $S^2$, which is an $M(\mathbb Z,2)$.  Sweney~\cite{homotopysmall} constructs a finite semigroup with minimal ideal a rectangular band that is an $M(\mathbb C_2,2)$, where $C_n$ denotes the cyclic group of order $n$.  It remains to construct a finite semigroup with minimal ideal a rectangular band that is an $M(C_n,2)$ with $n\geq 3$, which we proceed to do.

\section{The semigroup $S_n$ yielding an $M(C_n,2)$} 
Because the crux of the proof relies on the existence of a semigroup $S_C$ for each cyclic group $C$ as described in the previous section, we proceed constructively: we describe the semigroup and compute its homology by hand to demonstrate that it does, indeed, yield a Moore Space $M(C,2)$.

Begin by fixing $n$ and let $K_n$ be the $(n+2)$-by-$n$ rectangular band \[K_n=\{(i,j):i\in\{x,y,0,...,n-1\},j\in\{0,...,n-1\}\}\] with the operation $(i,j)(k,l)=(i,l)$.

We view the cyclic group $C_n$ of order $n$ as integers modulo $n$, and so throughout this section one should always understand all integers as being taken modulo $n$.
Next, put $T_n=C_n\times\{s,t\}$, with the operation $(k,z)(l,w)=(k+l,w)$. Now we define our semigroup $S_n = K_n\sqcup T_n$ with the following multiplication:

For $(i,j)\in K_n$ and $(k,z)\in T_n$, we define $(i,j)(k,z)=(i,k+j)$.

On the other hand, for multiplications  $T_n\cdot K_n$, we put $(k,z)(x,i)=(x,i)$, $(k,z)(l,i)=(k+l,i)$, where $l\in\{0,\ldots,n-1\}$,  $(k,s)(y,i)=(k-1,i)$, and $(k,t)(y,i)=(k,i)$.

Note that in what follows  $\mathbb ZK_n$ is a projective $\mathbb ZM_n$-module, as it is $\bigoplus_{i=0}^{n-1}\mathbb ZM_n(x,i)$.

\begin{Thm}\label{t:moore.cyclic}
    The classifying space of $M_n=S_n^\mathbb{1}$ is a Moore space $M(C_n,2)$.
\end{Thm}
\begin{proof}
We use Theorem~\ref{t:top.interp.homology} to compute the homology of the classifying space.
We begin by obtaining a projective resolution of the trivial module $\mathbb{Z}$ for $M_n$. Consider the augmentation map $\mathbb{Z}M_n(x,0)\xrightarrow\varepsilon \mathbb{Z}\to 0$. Notice that $(k,t)[(y,0)-(x,0)]=(k,0)-(x,0)$, so $\ker{\varepsilon}$ is generated by $(y,0)-(x,0)$. Since our goal is to obtain an exact sequence, we define a map $\mathbb{Z}M_n \xrightarrow{\varphi} \mathbb{Z}M_n(x,0)$ with $\varphi(\sum_i k_im_i)=\sum_i k_im_i[(y,0)-(x,0)]$, i.e., $\varphi$ is right multiplication by $(y,0)-(x,0)$, so that $\mathop{\mathrm{im}}\varphi=\ker\varepsilon$. 

Next, $\ker\varphi \supseteq K_n$, since $(i,j)[(y,0)-(x,0)]=(i,0)-(i,0)=0$ for any $(i,j)\in K_n$. Also, $(k,t)-(k+1,s)\in \ker\varphi$ for all $k\in C_n$, since \[[(k,t)-(k+1,s)][(y,0)-(x,0)]=(k,0)-(k,0)+(x,0)-(x,0)=0.\] Since all of $T_n$ fixes $(x,0)$ and, for a fixed $k$, $(k,t)$ and $(k+1,s)$ are the only elements of $T_n$ with $(k,t)(y,0)=(k+1,s)(y,0)$, $\mathbb{Z}K_n$ and the $(k,t)-(k+1,s)$ generate $\ker\varphi$. Then see that $(0,t)-(1,s)$ in turn generates $\{(k,t)-(k+1,s): k\in C_n\}$ as a $\mathbb{Z}M_n$-module. Finally, observe that $(0,t)$ fixes the generator $(0,t)-(1,s)$.

So, set the next module in our resolution to \[\mathbb{Z}M_n(0,t)\oplus \mathbb{Z}K_n =\mathbb{Z}M_n(0,t)\oplus \bigoplus_{i=0}^{n-1} \mathbb{Z}M_n(x,i),\] and equip it with the map $\mathbb{Z}M_n(0,t)\oplus \mathbb{Z}K_n\xrightarrow{\psi}\mathbb{Z}M_n$ given by \[\psi(c(0,t),0)=c[(0,t)-(1,s)],\ \ \ \psi(0, \sum_{i}a_i(x,i))=\sum_i a_i(x,i),\] for $0 \leq  i \leq {n-1}$ so that $\mathop{\mathrm{im}} \psi=\ker\varphi$.

Fix $w\in \{x,y,0,\ldots,n-1\}$ and compute: \begin{align*}
    \psi((w,i)(0,t), (w,i+1)-(w,i))&=(w,i)[(0,t)-(1,s)]+(w,i+1)-(w,i)\\
    &=(w,i)-(w,i+1)+(w,i+1)-(w,i)\\
    &=0,
\end{align*} so $((w,i)(0,t), (w,i+1)-(w,i)) \in \ker\psi$. This is clearly the only kind of element in the kernel (indeed, $\psi$ takes the summand $\mathbb ZT_n$ injectively to  $\mathbb ZT_n$), so we put $\mathbb{Z}K_n$ as our next projective module equipped with the injection \[0\to\mathbb{Z}K_n\xrightarrow{\xi}\mathbb{Z}M_n(0,t)\oplus \mathbb{Z}K_n\] defined by $\xi(w,i)=((w,i)(0,t),(w,i+1)-(w,i)),$ so that $\mathop{\mathrm{im}} \xi=\ker\psi$.

Our full projective resolution is therefore \[0\longleftarrow \mathbb{Z} \xleftarrow{\,\,\varepsilon\,\,} \mathbb{Z}M_n(x,0) \xleftarrow{\,\,\varphi\,\,} \mathbb{Z}M_n \xleftarrow{\,\,\psi\,\,} \mathbb{Z}M_n(0,t)\oplus \mathbb{Z}K_n \xleftarrow{\,\,\xi\,\,} \mathbb{Z}K_n \longleftarrow 0. \]

Deleting the leftmost $\mathbb Z$ and tensoring with the right trivial module $\mathbb{Z}$ yields the chain complex \[0 \longleftarrow \mathbb{Z}\xleftarrow{\,\,0\,\,} \mathbb{Z} \xleftarrow{\,\,A\,\,}\mathbb{Z}^{n+1}\xleftarrow{\,\,B\,\,} \mathbb{Z}^n \longleftarrow{} 0,\] where \[A=\begin{bmatrix}
  0 \\
  1 \\
  1 \\
  \vdots \\
  1 \\
\end{bmatrix}\ \text{and}\ B=\begin{bmatrix}
1 & -1 & 1 & 0&\cdots & 0 & 0 \\
1 & 0 & -1 & 1& \ddots & 0 & 0 \\
1 & 0 & 0 & -1&\ddots & 0 & 0 \\
\vdots & \vdots & \vdots &0& \ddots & \vdots & \vdots \\
1 & 0 & 0 & 0&\ddots & -1 & 1 \\
1 & 1 & 0 & 0&\cdots &  0 & -1
\end{bmatrix},\]
and both matrices are right multiplied against row vectors in the appropriate module.

Finally, we compute homology. Clearly, $H_0=\mathbb{Z}$. Since $A$ is surjective, $H_1=\ker{0}/\mathop{\mathrm{im}}{A}=\mathbb{Z}/\mathbb{Z}=0$.

To compute $H_2$, notice first that the sum of the rows of $B$ is $(n, 0, 0,\ldots)$, so we can replace the last row with $(n, 0, 0, \ldots)$ without changing its row space. Denote this matrix $B'$: \[B'=\begin{bmatrix}
1 & -1 & 1 & 0&\cdots & 0 & 0 \\
1 & 0 & -1 & 1& \ddots & 0 & 0 \\
1 & 0 & 0 & -1&\ddots & 0 & 0 \\
\vdots & \vdots & \vdots &0& \ddots & \vdots & \vdots \\
1 & 0 & 0 & 0&\ddots & -1 & 1 \\
n & 0 & 0 & 0&\cdots &  0 & 0
\end{bmatrix}.\]

Note now that $w=(w_0,w_1,\ldots, w_n)\in \ker A$ if and only if $\sum_{i=1}^nw_i=0$.  So $\ker A$ has as basis the first $n-1$ rows of $B'$ along with $(1, 0,\ldots,0)$. Thus, \[H_2\cong (\mathbb{Z}\oplus \cdots \oplus \mathbb{Z})/(n\mathbb{Z}\oplus \cdots\oplus \mathbb{Z})=\mathbb{Z}/n\mathbb{Z}.\]

%For $v=(v_1,\ldots,v_n)\in \mathbb{Z}^n$,  let $w=(w_0,w_1,\ldots,w_n)=vB\in \ker A$. Then \[w_i=\begin{cases} 
      %\sum_{k=1}^n v_k & i=0 \\
    %  v_n-v_1 & i = 1 \\
     % v_i-v_{i+1} & 2 \leq i\leq n. 
 %  \end{cases} \] Set $v_n=t$, so that \begin{align*}
%v_{n-1} &= v_n + w_{n-1} = t + w_{n-1} \\
%v_{n-2} &= v_{n-1} + w_{n-2} = t + w_{n-1} + w_{n-2} \\
%&\vdots \\
%v_1 &= t + \sum_{i=1}^{n-1} w_i,
%\end{align*} and the first coordinate becomes
%$w_0 = \sum_{k=1}^n v_k = nt + \sum_{i=1}^{n-1} i \cdot w_i$.

%Since $t\in \mathbb{Z}$ is arbitrary (regardless of the value of $t$, we see that $\sum_{i=1}^n w_i =0$), we see that $w\in \text{im}\ B$ precisely when $w_0$ is congruent to $\sum_{i=1}^{n-1} i \cdot w_i$ modulo $n$. On the other hand, $w_0$ is a free parameter if we only require that $w\in \ker A$. Thus, $H_2=\mathbb{Z}/n\mathbb{Z}$.

Lastly, $H_n=0$ for all $n\geq 3$ (in particular, $H_3=0$ because $B$ is injective).
\end{proof}

%\textbf{To do:  Put the previous sections together to prove that we can realize all finite joins of Moore spaces of finitely generated abelian groups.  Mention that we have proved Theorems~A and~B.}
We can now prove Theorem~B, which in turn implies Theorem~A.

\begin{Thm}
    Let $A_2,\ldots, A_n$ be a finite sequence of finitely generated groups.  Then there is a finite semigroup $S$ with simply connected classifying space that is homotopy equivalent to the wedge of Moore spaces $\bigvee_{i=2}^nM(A_i,i)$.
\end{Thm}
\begin{proof}
If $A$ is a finitely generated abelian group, then $A\cong \mathbb Z^n\oplus C_{i_1}\oplus\cdots\oplus C_{i_k}$ where the $C_{i_j}$ are finite cyclic groups of order $i_j$.  By Fiedorowicz's observation (see~\cite{classifyingspacecompletelysimple} for proof), we can realize up to homotopy $S^2$ (an $M(\mathbb Z,2)$) as the classifying space of the $2\times 2$ rectangular band with an adjoined identity.  Theorem~\ref{t:moore.cyclic} lets us realize an $M(C_{i_j},2)$ as the classifying space  of a finite monoid with a minimal ideal that is a rectangular band.  Thus we can realize an $M(A,2)$ as the classifying space of a finite monoid with minimal ideal by Theorem~\ref{t:wedge}.  Using Sweeney's result~\cite{homotopysmall}, we can realize the homotopy type of the suspensions of the classifying space of any finite monoid with a finite monoid whose minimal ideal is a rectangular band.  Therefore, we can find a finite monoid with minimal ideal a rectangular band which is an  $M(A,i)$ for any $i\geq 2$.  The result then follows by another application of Theorem~\ref{t:wedge}.
\end{proof}

\section{Global dimension of regular monoids and some questions of Sweeney}\label{s:globdim}
Let $K$ be a commutative ring with unit and $M$ a finite regular monoid.  We provide upper and lower bounds on the (left) global dimension of $KM$, generalizing Nico's work over a field~\cite{Nico1,Nico2}.  We will use these results to prove two conjectures about the homology of classifying spaces of finite semigroups from~\cite{homotopysmall}.  First, we need some lemmas proved in~\cite{quasistrat} for algebras over a field.  % For convenience we put $\gldim 0=-1$.

We follow here the theory of stratifying ideals~\cite{CPS2}, strongly idempotent ideals~\cite{auslanderidem} and homological epimorphisms~\cite{homologicalepi}.  Isbell's zig-zag theorem says that a ring homomorphism $\p\colon R\to S$ is an epimorphism if and only if the multiplication map $S\otimes_R S\to S$ is an $S$-bimodule isomorphism.  This occurs in particular if $\p$ is onto.  Note that this implies that, for any left $S$-module $A$, the natural map $S\otimes_R A\to A$ is an isomorphism, as $S\otimes_R A\cong S\otimes_R (S\otimes_S A)\cong (S\otimes_R S)\otimes_S A\cong S\otimes_S A\cong A$, and this composition of isomorphisms is the natural map.

A ring homomorphism $\p\colon R\to S$ is called a \emph{homological epimorphism}~\cite{homologicalepi} if, for every right $S$-module $A$ and left $S$-module $B$, the natural map \[\Tor_i^R(A,B)\to \Tor_i^S(A,B)\] is an isomorphism for all $i\geq 0$.  There are a number of equivalent formulations of this property.  The following is extracted from~\cite[Theorem~4.4]{homologicalepi}.

\begin{Thm}[Geigle-Lenzing]
Let $\p\colon R\to S$ be a ring homomorphism.  Then the following are equivalent.
\begin{enumerate}
  \item $\p$ is a homological epimorphism.
  \item The natural map \[\Ext^i_S(A,B)\to \Ext^i_R(A,B)\] is an isomorphism for all left $S$-modules $A,B$ and $i\geq 0$.
  \item $\p$ is an epimorphism and $\Tor_i^R(S,S)=0$ for all $i\geq 1$.
  \item $\p$ is an epimorphism and  $\Ext^i_R(S,S)=0$ for all $i\geq 1$.
\end{enumerate}
\end{Thm}

For the case of finite dimensional algebras, there has been much study of ideals $I$ with the property that  $R\to R/I$ is a homological epimorphism~\cite{auslanderidem,CPS2}.  One says that $I$ is a \emph{homological ideal} if $R\to R/I$ is a homological epimorphism. (The notion of a stratifying ideal~\cite{CPS2} is essentially the same thing for finite dimensional algebras.)
Most of the results work in full generality, but we sketch some proofs for completeness.  The following fact is well known (cf.~\cite[Exercise~19, p.~126]{CartanEilenberg}).

\begin{Prop}\label{p:idem.ideal.nec}
Let $I\lhd R$ be an ideal.  Then $\Tor_1^R(R/I,R/I)\cong I/I^2$, and so $\Tor_1^R(R/I,R/I)=0$ if and only if $I=I^2$.
\end{Prop}

 The next result is well known in the context of finite dimensional algebras~\cite{CPS2,auslanderidem}.

\begin{Prop}\label{p:standard.strat}
Let $R$ be a ring and $I\lhd R$ an idempotent ideal which is flat as a left or right $R$-module.  Then $I$ is a homological ideal.
\end{Prop}
\begin{proof}
Obviously, $R\to R/I$ is an epimorphism.
Since $I=I^2$, Proposition~\ref{p:idem.ideal.nec} shows that $\Tor_1^R(R/I,R/I)=0$. Without loss of generality, assume that $I$ is flat as a left $R$-module.   If $i\geq 2$, the exact sequence $0\to I\to R\to R/I$ yields an exact sequence $0\to \Tor_i^R(R/I,R/I)\to \Tor_{i-1}^R(R/I,I)\to 0$, and hence $\Tor_i^R(R/I,R/I)\cong \Tor_{i-1}^R(R/I,I)=0$, as $I$ is flat.  Thus $I$ is a homological ideal.
\end{proof}

Our next lemma is known to experts.

\begin{Lemma}\label{l:eRisproj}
Let $R$ be a ring and $e\in R$ an idempotent.  Suppose that $ReR$ is a projective left $R$-module.   Then $eR$ is a projective left $eRe$-module, and so $eQ$ is a projective $eRe$-module for any projective left $R$-module $Q$.
\end{Lemma}
\begin{proof}
Let $P=\bigoplus_{x\in eR} Re$.  Then $P$ is a projective $R$-module and we have a surjective $R$-module homomorphism $\Psi\colon P\to ReR$ sending $\sum_{x\in eR}r_x$ to $r_xx\in ReR$.  Since $ReR$ is projective, $\Psi$ must split, and so $P\cong ReR\oplus P'$.  Therefore $eR=eReR$ is a direct summand in $eP\cong \bigoplus_{x\in eR} eRe$.  We conclude that $eR$ is projective as a left $eRe$-module.
\end{proof}

The next lemma is a straightforward exercise in dimension shifting.

\begin{Lemma}\label{l:pdim.quot}
Let $R$ be a ring and $I$ an ideal.  Let $A$ be an $R/I$-module.  Then $\pdim_R A\leq \pdim_{R/I} A+\pdim_R I+1$.
\end{Lemma}
\begin{proof}
If either of $\pdim_{R/I} A$ or $\pdim_R I$ is infinite, there is nothing to prove.  So we assume these are finite and proceed by induction on  $\pdim_{R/I} A$. Put $s=\pdim_R I$.   First note that $\pdim_R R/I\leq s+1$ because of the exact sequence $0\to I\to R\to R/I\to 0$.  It follows that if $P$ is a projective $R/I$-module, then $\pdim_R P\leq s +1$.  Suppose that the result holds when $\pdim_{R/I} B=r\geq 0$ and that $\pdim_{R/I} A=r+1$.  Then there is an exact sequence $0\to B\to P\to A\to 0$ of $R/I$-modules where $P$ is a projective $R/I$-module and $\pdim_{R/I} B=r$.  For every $R$-module $C$ we have an exact sequence \[0=\Ext_R^{r+s+2}(B,C)\to \Ext_R^{r+s+3}(A,C)\to \Ext_R^{r+s+3}(P,C)=0\] where we have used that $\pdim_R P\leq s+1$ and $\pdim_R B\leq r+s+1$.   Thus $\pdim_R A\leq r+s+2$.
\end{proof}

In what follows $\gldim R$ will denote the left global dimension of a ring $R$.  Note that for a noetherian ring like $KM$ with $M$ a finite monoid and $K$ a commutative noetherian ring, one has that the left and right global dimensions coincide.  In any event, all of the results here about left global dimension have analogues for right global dimension.

\begin{Prop}\label{p:cornergldim}
Suppose that $R$ is a ring, $e\in R$ is an idempotent and $ReR$ is a projective left $R$-module.  Then
\begin{enumerate}
\item $\gldim eRe\leq \gldim R$.
\item $\gldim R/ReR\leq \gldim R$.
\item $\gldim R\leq \gldim eRe+\gldim R/ReR+2$.
\end{enumerate}
\end{Prop}
\begin{proof}
For the first item, let $M$ be an $eRe$-module.  Then $\ov M=Re\otimes_{eRe} M$ is an $R$-module and $e\ov M=e(Re\otimes_{eRe}M)\cong eRe\otimes _{eRe} M\cong M$ as $eRe$-modules.  Let $P_\bullet\to \ov M$ be a projective resolution of length at most $\gldim R$.  Then by Lemma~\ref{l:eRisproj} and exactness of $N\mapsto eN$, we have that  $eP_\bullet\to M$ is a projective resolution of length at most $\gldim R$.  Thus $\gldim eRe\leq \gldim R$.

Since $R\to R/ReR$ is a homological epimorphism by Proposition~\ref{p:standard.strat}, if $A,B$ are $R/ReR$-modules and $n>\gldim R$, then $\Ext^n_{R/ReR}(A,B)\cong \Ext^n_R(A,B)=0$.  Thus $\gldim R/ReR\leq \gldim R$ by Lemma~\ref{l:brown}.

For the final item, if either $\gldim eRe$ or $\gldim R/ReR$ is infinite, then there is nothing to prove.  Else, assume $\gldim eRe=r$ and $\gldim R/ReR=s$.  Let $A$ be an $R$-module.  Construct an exact sequence \[0\to \Omega_r\to P_{r-1}\to\cdots\to P_0\to A\to 0\] with $P_0,\ldots, P_{r-1}$ projective.  Then we have that \[0\to e\Omega_r\to eP_{r-1}\to\cdots\to eP_0\to eA\to 0\] is exact and, by Lemma~\ref{l:eRisproj}, $eP_0,\ldots, eP_0$ are projective $eRe$-modules.  Since $\pdim_{eRe} eA\leq r$, it follows that $e\Omega_r$ is a projective $eRe$-module by Lemma~\ref{l:brown}.  Since $Re$ is a projective $R$-module, it follows that $Re\otimes_{eRe} e\Omega_r$ is a projective $R$-module.

The multiplication map $\Phi\colon Re\otimes_{eRe}e\Omega_r\to Re\Omega_r$ is surjective and induces an isomorphism $e\Phi\colon e(Re\otimes_{eRe}e\Omega_r)\to e\Omega_r$.  It follows that if $L=\ker \Phi$, then $eL=0$.  Therefore, $L$ is an $R/ReR$-module.  So $\pdim_R L\leq s+1$ by Lemma~\ref{l:pdim.quot}.  Then the exact sequence
\[0\to L\to Re\otimes_{eRe} e\Omega_r \to Re\Omega_r\to 0\] and projectivity of the middle term yields that $\pdim_R Re\Omega_r\leq s+2$.

The exact sequence
\[0\to Re\Omega_r\to \Omega_r\to \Omega_r/Re\Omega_r\to 0\] then yields
\[\pdim_R\Omega_r\leq \max\{\pdim_R Re\Omega_r,\pdim_R \Omega_r/Re\Omega_r\}\leq \max\{s+2,s+1\}= s+2\] by Lemma~\ref{l:pdim.quot} and the above paragraph.  It follows that $\pdim_R A\leq r+s+2$, as required.
\end{proof}

The next result generalizes a standard bound for the global dimension of quasihereditary algebras.

\begin{Cor}\label{c:heredity.chain}
Let $1=e_0,e_1,\ldots, e_n=0$ be idempotents of $R$ such that if $I_j=Re_jR$, then \[0=I_n\subseteq I_{n-1}\subseteq\cdots\subseteq I_1\subseteq I_0=R\] and $I_j/I_{j+1}$ is a projective left $R/I_{j+1}$-module for $0\leq j\leq n-1$.  Then the inequalities \[\gldim (e_jRe_j/e_jI_{j+1}e_j)\leq \gldim R\leq \sum_{i=0}^{n-1}\gldim (e_iRe_i/e_iI_{i+1}e_i)+2(n-1)\] hold for all $0\leq j\leq n$.
\end{Cor}
\begin{proof}
We induct on $n$.  If $n=1$, then $I_0=R$, $I_1=0$ and $R=e_0Re_0/e_0I_1e_0$, whence the inequalities are actually equalities.  Assume the result holds for $n$, and we prove it for $n+1$.  Then applying the inductive hypotheses to $R/I_n$ with respect to the idempotents $e_0+I_n,\ldots, e_n+I_n$, we see that \[\gldim (e_kRe_k/e_kI_{k+1}e_k)\leq \gldim R/I_n\leq \sum_{i=0}^{n-1} \gldim (e_iRe_i/e_iI_{i+1}e_i)+2(n-1)\] for all $0\leq k\leq n$. Applying Proposition~\ref{p:cornergldim} to $R$ and $e_n$ yields
\begin{align*}
\gldim R &\leq \gldim e_nRe_n+\gldim R/I_n+2 \\ &\leq \sum_{i=0}^{n} \gldim (e_iRe_i/e_iI_{i+1}e_i)+2n.
\end{align*}
Also since $\gldim R/I_n, \gldim e_nRe_n\leq \gldim R$ by Proposition~\ref{p:cornergldim}, we see that the lower bounds hold.
 \end{proof}

In order to apply this to monoid algebras, we shall need the following well-known result.   

\begin{Prop}\label{p:indep.of.base}
Let $K$ be a commutative ring with unit, $M$ a monoid and $V$ a $KM$-module.  Then, as abelian groups, $H_n(M,V)\cong \Tor_n^{KM}(K,V)$ and $H^n(M,V)\cong \Ext^n_{KM}(K,V)$ where $K$ has the trivial $K$-module structure.
\end{Prop}
\begin{proof}
We just handle the first statement, as the proof of the second is similar.  Let $F_\bullet\to \mathbb Z$ be a free resolution of the trivial right $\mathbb ZM$-module $\mathbb Z$.  Note that $F_\bullet \otimes_{\mathbb Z}K\to \mathbb Z\otimes_{\mathbb Z} K\cong K$ is then a free resolution of $K$ as a $KM$-module.  The reason this chain complex is exact is because $\mathbb ZM$ is a free abelian group, and so $H_n(F_\bullet\otimes_{\mathbb Z} K)\cong \Tor^{\mathbb Z}_n(\mathbb Z,K)=0$ for $n\geq 1$.  Thus, we have
$\Tor_n^{KM}(K,V)\cong H_n((F_{\bullet}\otimes_{\mathbb Z}K)\otimes_{KM} V)\cong H_n(F_{\bullet}\otimes_{\mathbb ZM} V)\cong \Tor^{\mathbb ZM}_n(\mathbb Z,V)=H_n(M,V)$.  Here we used that for any $\mathbb ZM$-module $A$, there is a natural isomorphism $(A\otimes_{\mathbb Z} K)\otimes_{KM} V\to A\otimes_{\mathbb ZM} V$ sending $a\otimes c\otimes v$ to $a\otimes cv$ with inverse $a\otimes v\mapsto a\otimes 1\otimes v$.
\end{proof}

Recall that if $G$ is a group and $A,B$ are left $KG$-modules, then $A\otimes_K B$ and $\Hom_K(A,B)$ are left $KG$-modules, the former via $g(a\otimes b) =ga\otimes gb$ and the latter via $(gf)(a) = gf(g\inv a)$.  Moreover, one has an isomorphism $\Hom_{KG}(A\otimes_K B,C)\cong \Hom_{KG}(A,\Hom_K(B,C))$, natural in $A$, $B$ and $C$, taking $p\colon A\otimes_K B\to C$ to $\Phi$ with $\Phi(a)(b) = \p(a\otimes b)$.  In particular, $\Hom_{KG}(B,C)\cong \Hom_{KG}(K,\Hom_K(B,C))$, naturally in $B$ and $C$.  Let $P_\bullet\to A$ be a projective $KG$-module resolution.  Then we have a cochain complex of $KG$-modules $\Hom_K(P_\bullet,B)$ for any left $KG$-module $B$.  Since $KG$ is a free $K$-module, $P_\bullet\to A$ is a projective resolution of $A$ as a $K$-module.  Therefore, $\Ext^n_K(A,B)=H^n(\Hom_K(P_\bullet,B))$ is a $KG$-module, and the module structure is independent of the choice of $KG$-projective resolution of $A$.  The following spectral sequence is folklore.

\begin{Prop}\label{p:spseq}
Let $K$ be a commutative ring and $G$ a group. Let $A,B$ be $KG$-modules. Then there is a convergent spectral sequence \[\Ext^p_{KG}(K,\Ext_K^q(A,B))\Rightarrow \Ext_{KG}^{p+q}(A,B).\]
\end{Prop}
\begin{proof}
Let $P_\bullet\to A$ and $Q_\bullet\to K$ be free resolutions over $KG$.  Define a double cochain complex $C^{p,q}=\Hom_{KG}(P_p\otimes_k Q_q,B)$ with the obvious boundary maps.  We shall use the spectral sequence of a double complex~\cite{Rotmanhom}.  Note that $C^{p,\bullet} \cong \Hom_{KG}(P_p,\Hom_K(Q_\bullet,B))$.  Since $P_p$ is projective, we see that $H^q(C^{p,\bullet})\cong \Hom_{KG}(P_p, \Ext_K^q(K,B))$.  Since $K$ is a projective $K$-module, we see that this is nonzero only when $q=0$, in which case we obtain $\Hom_{KG}(P_p,B)$.
It follows that the spectral sequence in this direction collapses on the second page, yielding that $n^{th}$-cohomology of the total complex is $H^n(\Hom_{KG}(P_\bullet, B))=\Ext^n_{KG}(A,B)$.

 On the other hand, $C^{\bullet, p} \cong \Hom_{KG}(Q_p,\Hom_K(P_\bullet, B))$.  Since $Q_p$ is projective, we see that \[H^q(C^{\bullet, p})\cong \Hom_{KG}(Q_p,H^q(\Hom_K(P_\bullet, B))) = \Hom_{KG}(Q_p,\Ext^q_K(A,B)).\]  Thus the second page of the spectral sequence in this direction is given by $\Ext_{KG}^p(K,\Ext^q_{KG}(A,B))$, as required.
\end{proof}

The following consequence is folklore.

\begin{Cor}\label{c.group.gldim}
Let $K$ be a commutative ring with unit and $G$ a group.  Then $\pdim_{KG} K,\gldim K\leq \gldim KG\leq \gldim K+\pdim_{KG} K$.
\end{Cor}
\begin{proof}
Let $A$ be a $K$-module.  Since $KG$ is a free $K$-module, any projective $KG$-resolution of the trivial $KG$-module $A$ is a $K$-projective resolution of $A$, and so $\pdim_K A\leq \pdim_{KG} A\leq \gldim KG$.  The first inequality follows.  The second inequality follows from Proposition~\ref{p:spseq}, since if $p+q>\gldim K+\pdim_{KG} K$, then $\Ext_{KG}^p(K,\Ext^q_K(A,B))=0$, as either $p>\pdim_{KG} K$ or $q>\gldim K$.  The result follows.
\end{proof}

The following is well known, but we give a proof for completeness.

\begin{Cor}\label{c:group.case}
Let $G$ be a finite group and $K$ a commutative ring.  If $|G|$ is a unit in $K$, then $\gldim KG=\gldim K$, otherwise $\gldim KG=\infty$.
\end{Cor}
\begin{proof}
If $|G|$ is a unit in $K$, then $e=\frac{1}{|G|}\sum_{g\in G}g$ is an idempotent with $KGe\cong K$.  Thus $K$ is a projective $KG$-module and the equality follows from Corollary~\ref{c.group.gldim}.  If $|G|$ is not a unit in $K$, then we can find a prime divisor $p$ of $|G|$ which is not a unit in $K$.  Let $C$ be a cyclic subgroup of $G$ of order $p$.  By Shapiro's lemma~\cite[Proposition~III.6.2]{Browncohomology}, we have that $H_n(C,K)\cong H_n(G,KG\otimes_{KC} K)$.  Now it is well known~\cite[Section~III.1, page 58]{Browncohomology} that $H_n(C,K) = K/pK$ for $n$ odd.  Since $p$ is not a unit in $K$, we must have that $pK\neq K$.  It follows from Proposition~\ref{p:indep.of.base} that $\Tor^{KG}_n(K,KG\otimes_{KC} K)$ is nonzero for all odd numbers $n$, and hence $\pdim_{KG} K=\infty$.  Therefore, $\gldim KG=\infty$.
\end{proof}

The following lemma shows that a $0$-simple $0$-minimal left ideal is homological.

\begin{Lemma}\label{l:min.ideal}
Let $M$ be a finite monoid with zero and let $I$ be a $0$-simple $0$-minimal ideal.  Then $K_{\mathbb 0}I=K_{\mathbb 0}[MeM]$ for an idempotent $e\in I$, and $K_{\mathbb 0}I$ is a projective left $K_{\mathbb 0}M$-module for any commutative ring $K$.
\end{Lemma}
\begin{proof}
If $e$ is any nonzero idempotent of $I$, then $MeM=I$ by $0$-minimality.
Let $f_1,\ldots, f_r$ be idempotent generators of the distinct nonzero principal left ideals of $I$.  Then $I\setminus \{0\}=\bigsqcup_{i=1}^r Mf_i\setminus \{0\}$.   It follows that $K_{\mathbb 0}I=\bigoplus_{i=1}^r K_{\mathbb 0}Mf_i$, and hence is projective.
\end{proof}

We now prove our main result for this section, generalizing a result of Nico~\cite{Nico1} for the case that $K$ is a field.

\begin{Thm}\label{t:nico.gen}
Let $M$ be a finite nontrivial regular monoid with zero and let \[0=Me_nM\subsetneq \cdots\subsetneq Me_1M\subsetneq Me_0M=M\] be a principal series for $M$.  Let $K$ be a commutative ring with unit.  If $|G_{e_i}|$ is not a unit in $K$ for some $i$, then $\gldim K_{\mathbb 0}M=\infty$.  Otherwise, $\gldim K_{\mathbb 0}M\leq 2(n-1)+n\cdot \gldim K$.
\end{Thm}
\begin{proof}
Let $I_j = K_{\mathbb 0}Me_jM$.  Note that $e_jKMe_j/e_jI_{j+1}e_j\cong KG_{e_j}$.  The result will then follow from Corollary~\ref{c:group.case} and Corollary~\ref{c:heredity.chain} once we prove that $I_j/I_{j+1}$ is projective as a $K_{\mathbb 0}[M/I_{j+1}]$-module.  But this follows from Lemma~\ref{l:min.ideal} applied to the $0$-minimal ideal $I_j/I_{j+1}$ of $M/I_{j+1}$.
\end{proof}

Theorem~\ref{t:nico.gen} immediately implies Theorem~C from the introduction.

\begin{Rmk}
As in~\cite{Nico2}, one can improve the upper bound in Theorem~\ref{t:nico.gen} to replace $n$ by the length of the longest chain of principal ideals in $M$.  We leave the details to the interested reader.
\end{Rmk}

\begin{Cor}\label{c:finite.gldim}
Let $M$ be a finite aperiodic regular monoid.  Then $\gldim \mathbb ZM$ is finite.  In fact, it is bounded above by $3n-2$ where $n$ is the number of $\mathscr J$-classes of $M$.  Consequently, $H_n(M,\mathbb Z)=0$ for $n$ sufficiently large.  Moreover, $\mathbb Z$ has a finite projective resolution by finitely generated projective $\mathbb ZM$-modules.
\end{Cor}
\begin{proof}
The first statement is immediate from Theorem~\ref{t:nico.gen}.  Let $k=\pdim_{\mathbb ZM}\mathbb Z$. Note that since $M$ is finite, $\mathbb ZM$ is left noetherian, and hence every finitely generated $\mathbb ZM$-module has a resolution by finitely generated projective modules.  Thus there is a projective resolution $P_\bullet\to \mathbb Z$ with $P_n$ finitely generated for all $n\geq 0$.  Then we have an exact sequence $0\to K\to P_{k-1}\to \cdots \to P_0\to \mathbb Z\to 0$ with $K$ finitely generated, by noetherianity, and projective by Lemma~\ref{l:brown} as $k=\pdim_{\mathbb ZM}\mathbb Z$.  This completes the proof.
\end{proof}

Since every finite band is an aperiodic regular monoid, the above result answers Question~8.4(b)--(c) of~\cite{homotopysmall}.  The next result answers Question~8.6 from~\cite{homotopysmall}.  We remark that if $M$ is a finite monoid, then $H_n(M,A)$ is a finitely generated abelian group for any finitely generated abelian group $A$.  Indeed, since $\mathbb ZM$ is noetherian, $A$ has a resolution $F_\bullet\to A$ by finitely generated free modules.  Then $\mathbb Z\otimes F_\bullet$ is a chain complex of finitely generated free abelian groups.  Therefore, $H_n(M,A) = H_n(\mathbb Z\otimes F_\bullet)$ is finitely generated.

\begin{Cor}\label{c:torsion.regular}
Let $M$ be a finite regular monoid.  Then $H_n(M,\mathbb Z)$ is finite for all $n$ sufficiently large.
\end{Cor}
\begin{proof}
Since $\mathbb Q$ is a flat $\mathbb Z$-module, we have that $H_n(M,\mathbb Q)\cong H_n(M,\mathbb Z)\otimes_{\mathbb Z}\mathbb Q$ by the universal coefficient theorem.  It thus suffices to show that $H_n(M,\mathbb Q)=0$ for $n$ sufficiently large.  But $H_n(M,\mathbb Q)\cong \Tor_n^{\mathbb QM}(\mathbb Q,\mathbb Q)$ by Proposition~\ref{p:indep.of.base}, and $\mathbb QM$ has finite global dimension by the special case of Theorem~\ref{t:nico.gen} from~\cite{nico}.  Therefore, $\Tor_n^{\mathbb QM}(\mathbb Q,\mathbb Q)=0$ for $n$ sufficiently large.  We conclude that $H_n(M,\mathbb Z)\otimes_{\mathbb Z}\mathbb Q\cong H_n(M,\mathbb Q)\cong \Tor_n^{\mathbb QM}(\mathbb Q,\mathbb Q)=0$, and hence $H_n(M,\mathbb Z)$ is a finite group, for $n$ sufficiently large.
\end{proof}

We now present our algebraic proof that $BJ(S)$ has the homology of the suspension of $BS$ for a monoid $S$.

\begin{Prop}\label{p:susp.homology}
Let $S$ be a finite monoid.  Then $BJ(S)$ is simply connected and, for $n\geq 2$, $H_n(BJ(S))\cong H_{n-1}(BS)$.  In particular, if $BS$ is a Moore space of type $M(A,n)$, them $BJ(S)$ is a Moore space of type $M(A,n+1)$.
\end{Prop}
\begin{proof}
Since $J(S)$ has a minimal ideal which is a rectangular band, $\pi_1(BJ(S))$ is trivial.
Let $e=(1,1)$.  Then $K=SeS$ is the minimal ideal of $J(S)$ and $\mathbb ZK$ is a homological ideal of $\mathbb ZJ(S)$ by Proposition~\ref{p:standard.strat} and Lemma~\ref{l:min.ideal} (applied to $J(S)^{\mathbb 0}$).  There is a surjective homomorphism $\mathbb ZJ(S)\to \mathbb ZS$ that is the identity on $S$ and sends $K$ to $0$.  Denote by $\Inf A$ the inflation of a $\mathbb ZS$-module $A$ to $\mathbb ZJ(S)$.  Notice that there is a surjective module homomorphism $\mathbb ZMe\to \mathbb Z$ sending $(1,1)$ and $(2,1)$ to $1$.  The kernel is generated by $(2,1)-(1,1)$ as an abelian group, and as a left $\mathbb ZJ(S)$-module it is isomorphic to $\Inf \mathbb Z$.  Thus we have an exact sequence $0\to \Inf\mathbb Z\to \mathbb ZMe\to \mathbb Z\to 0$.   This yields that $H_n(M,\mathbb Z)\cong H_{n-1}(M,\Inf \mathbb Z)$ for $n\geq 2$ by the long exact sequence in homology as $\mathbb ZMe$ is projective.

 Let $I$ be the augmentation ideal of $\mathbb ZS$, that is, the kernel of the augmentation homomorphism $\mathbb ZS\to \mathbb Z$ given by $s\mapsto 1$.   There is an exact sequence of right $\mathbb ZS$-modules, $0\to I\to \mathbb ZS\to \mathbb Z\to 0$.  Hence $H_n(S,\mathbb Z)\cong H_{n-1}(S,I)$ for $n\geq 2$ and there is an exact sequence $0\to H_1(S,\mathbb Z)\to I\otimes_{\mathbb ZS}\mathbb Z \to \mathbb Z\to \mathbb Z\to 0$, whence $H_1(S)\cong I\otimes_{\mathbb ZS}\mathbb Z$.

 Now consider the exact sequence of right $\mathbb ZJ(S)$-modules $0\to R\to e\mathbb ZM\to \mathbb Z\to 0$, where $R$ has basis as an abelian group all differences $(1,s)-(1,1)$ with $s\neq 1$.  Notice that $K$ annihilates $R$ and that as a $\mathbb ZS$-module, $R\cong \Inf I$.  We deduce that $H_{n-1}(M,\Inf\mathbb Z)\cong \Tor^{\mathbb ZJ(S)}_{n-2}(\Inf I,\Inf \mathbb Z)$ for $n\geq 3$, and that there is an exact sequence
 \[0\to H_1(M,\Inf\mathbb Z)\to \Inf I\otimes_{\mathbb ZJ(S)} \Inf \mathbb Z\to e\Inf \mathbb Z\to \mathbb Z\otimes_{\mathbb ZJ(S)}\Inf \mathbb Z\to 0.\]
 Since $\mathbb ZK$ is a homological ideal, we deduce that, for $n\geq 3$, \[H_n(M)\cong \Tor^{\mathbb ZJ(S)}_{n-2}(\Inf I,\Inf \mathbb Z)\cong \Tor^{\mathbb ZS}_{n-2}(I,\mathbb Z)=H_{n-2}(S,I)\cong H_{n-1}(S).\]  Since $e\Inf\mathbb Z=0$, it follows that for $n=2$, $H_2(M)\cong H_1(M,\Inf\mathbb Z)\cong \Inf I\otimes_{\mathbb ZJ(S)} \Inf \mathbb Z\cong I\otimes_{\mathbb ZS} \mathbb Z\cong H_1(S)$.  This completes the proof.
\end{proof}

\bibliographystyle{abbrv}
\bibliography{standard2}

\end{document}